\documentclass[11pt]{article}
\pdfoutput=1
\usepackage[UKenglish]{isodate}
\usepackage[T1]{fontenc}
\usepackage[noBBpl]{mathpazo}
\usepackage{microtype}
\usepackage{multicol}
\usepackage[a4paper,left=2.8cm,right=2.8cm,top=2.7cm,bottom=2.7cm]{geometry}
\usepackage{amsmath,amssymb,amsfonts,amsthm,array,float,url,xpatch,stmaryrd}

\renewcommand{\:}{\colon}
\renewcommand{\leq}{\leqslant}
\renewcommand{\geq}{\geqslant}
\newcommand{\leqn}{\trianglelefteqslant}
\newcommand{\im}{\mathrm{im}}
\newcommand{\e}{\epsilon}
\newcommand{\F}{\mathbb{F}}
\newcommand{\Aut}{\mathrm{Aut}}
\newcommand{\Inn}{\mathrm{Inn}}
\newcommand{\GL}{\mathrm{GL}}
\newcommand{\PGL}{\mathrm{PGL}}
\newcommand{\PSL}{\mathrm{PSL}}
\newcommand{\Sp}{\mathrm{Sp}}
\newcommand{\PSU}{\mathrm{PSU}}
\newcommand{\PSp}{\mathrm{PSp}}
\newcommand{\Or}{\mathrm{O}}

\newtheoremstyle{shdefinition}{8pt}{4pt}{}{}{\bfseries\boldmath}{.}{0.3em}{} 
\newtheoremstyle{shplain}{8pt}{4pt}{\itshape}{}{\bfseries\boldmath}{.}{0.3em}{} 
\theoremstyle{shdefinition}
\newtheorem{definition}{Definition}[section]

\newtheorem*{example*}{Example}
\newtheorem*{acknowledgements*}{Acknowledgements}
\theoremstyle{shplain}

\newtheorem{shtheorem}{Theorem}
\newtheorem*{theorem*}{Theorem}
\newtheorem{corollary}[definition]{Corollary}
\newtheorem{shcorollary}[shtheorem]{Corollary}
\newtheorem{proposition}[definition]{Proposition}
\newtheorem{lemma}[definition]{Lemma}

\usepackage{titlesec}
\titlespacing*{\section}{0pt}{\baselineskip}{0pt}
\titlespacing*{\subsection}{0pt}{0.66\baselineskip}{0pt}

\usepackage[shortlabels]{enumitem}
\setlist{leftmargin=0.8cm,topsep=0pt,itemsep=-2pt}
\setlist[enumerate]{label=\rm{(\roman*)}}

\numberwithin{equation}{section}

\renewenvironment{thebibliography}[1]
{ \begin{oldthebibliography}{#1}
  \setlength{\parskip}{0pt}
  \setlength{\itemsep}{2pt plus 0.3ex}
  \bgroup\footnotesize }
{ \egroup \end{oldthebibliography} }

\makeatletter
\renewenvironment{proof}[1][\proofname]{\par
  \pushQED{\qed}%
  \normalfont
  \topsep2pt \partopsep1pt 
  \trivlist
  \item[\hskip\labelsep
        \itshape
    #1\@addpunct{.}]\ignorespaces
}{%
  \popQED\endtrivlist\@endpefalse
  \addvspace{6pt plus 6pt}
}
\makeatother

\makeatletter
\g@addto@macro\normalsize{%
  \setlength\abovedisplayskip{0.4\baselineskip plus 0.4\baselineskip}
  \setlength\belowdisplayskip{0.4\baselineskip plus 0.4\baselineskip}
  \setlength\abovedisplayshortskip{-0.3\baselineskip}
  \setlength\belowdisplayshortskip{0.4\baselineskip plus 0.4\baselineskip}
}
\makeatother

\makeatletter\def\blfootnote{\gdef\@thefnmark{}\@footnotetext} \makeatother
\newcommand{\dateline}[1]{\enlargethispage{18pt}\blfootnote{\phantom{\Large M}\hspace{-1em}\hspace{-20pt}\emph{Date} #1}}

\begin{document}
 
\begin{center} 
{\LARGE \textbf{Representations of extensions of simple groups}} \\[11pt]
{\Large Scott Harper \& Martin W. Liebeck}                       \\[22pt]
\end{center}

\begin{center}
\begin{minipage}{0.8\textwidth}
\small
Feit and Tits (1978) proved that a nontrivial projective representation of minimal dimension of a finite extension of a finite nonabelian simple group $G$ factors through a projective representation of $G$, except for some groups of Lie type in characteristic~$2$; the exact exceptions for $G$ were determined by Kleidman and Liebeck (1989). We generalise this result in two ways. First we consider all low-dimensional projective representations, not just those of minimal dimension. Second we consider all characteristically simple groups, not just simple groups.
\end{minipage}
\end{center}
\vspace{-11pt}

\dateline{27 May 2024}

\section{Introduction} \label{s:intro}

Throughout this paper, by an \emph{extension of a group $G$ (by a group $N$)} we mean a surjective group homomorphism $\gamma\: H \to G$ (with kernel $N$). An extension $\gamma\: H \to G$ is said to be \emph{proper} if $\gamma$ is not an isomorphism and \emph{minimal} if $K\gamma < G$ for all $K < H$.

Let $k$ be an algebraically closed field. For a finite group $G$, let $R_k(G)$ be the minimal dimension of a nontrivial projective representation of $G$ over $k$, that is, a homomorphism $\lambda\: G \to \PGL_m(k)$.
In \cite{ref:FeitTits78}, Feit and Tits asked whether it was possible that an extension $H$ of a finite simple group $G$ could have a nontrivial projective representation $H \to \PGL_m(k)$ where $m < R_k(G)$. The following example highlights one way that this can happen.

\begin{example*} \label{ex:symplectic}
Assume that $\textrm{char}\, k \neq 2$. Let $G = \Sp_{2n}(2)$ with $n > 2$. By \cite[Theorem~(I)]{ref:FeitTits78}, there exists a minimal extension $H$ of $G$ (by an elementary abelian group of order $2^{2n}$) that embeds irreducibly in $\PGL_{2^n}(k)$. Hence, $R_k(H) \leq 2^n$, but $R_k(G) > 2^n$ (see the main theorem of \cite{ref:LandazuriSeitz74}).
\end{example*}

In light of this example, for a finite group $G$, define 
\[
n_G= \min\{ n \mid G \preccurlyeq \Sp_{2n}(2) \text{ irreducible} \}.
\]
The following is one of the main theorems of \cite{ref:FeitTits78}.

\begin{theorem*}[Feit \& Tits, 1978]
Let $G$ be a finite nonabelian simple group, and $\gamma\: H \to G$ a finite minimal extension. Let $\lambda\: H \to \PGL_m(k)$ be a nontrivial projective representation of minimal dimension. Assume that $m < 2^{n_G}$ if $\mathrm{char}\, k \neq 2$. Then $\ker\gamma = \ker\lambda$.
\end{theorem*}

Hence, for any finite nonabelian simple group $G$ the minimum of $R_k(H)$ across all finite extensions $H$ of $G$ is $\min\{R_k(G), 2^{n_G}\}$ if $\mathrm{char}\, k \neq 2$, and is $R_k(G)$ if $\mathrm{char}\, k = 2$. In \cite{ref:KleidmanLiebeck89}, Kleidman and Liebeck determine exactly when $2^{n_G} < R_k(G)$ and hence when there exists a finite extension $H$ of $G$ such that $R_k(H) < R_k(G)$.

Bounds on the smallest dimension of a nontrivial projective representation of a group have many applications. Moreover, often it is useful to know about all low-dimensional representations and not just the ones of minimal dimension. There are many such results for simple groups. Denote by ${\rm Lie}(p)$ the set of simple groups of Lie type defined over fields of characteristic $p$. Some results on low-dimensional representations of groups in ${\rm Lie}(p)$ in $p'$-characteristic are summarised in \cite{ref:Tiep}, while for representations of ${\rm Lie}(p)$ in characteristic $p$, see \cite{ref:Lubeck}, and for alternating groups, see \cite{ref:James}.

In view of this, it is natural to ask about the low-dimensional projective representations of extensions of simple groups. This is what our first main theorem concerns. In the statement, as usual, a projective representation $\lambda\: H \to \PGL_m(k) \cong \PGL(V)$ is \emph{imprimitive} if the preimage of $H$ in $\GL(V)$ permutes among themselves the subspaces $V_1$, \dots, $V_m$ in a direct sum decomposition $V = \bigoplus_{i=1}^m V_i$ where $m>1$; otherwise $\lambda$ is \emph{primitive}.

\begin{shtheorem} \label{thm:simple}
Let $G$ be a finite nonabelian simple group, and let $\gamma\: H \to G$ be a finite minimal extension. Let $\lambda\: H \to \PGL_m(k)$ be a primitive projective representation. Assume that $m < 2^{n_G}$ if $\mathrm{char}\, k \neq 2$. Then $\ker\gamma \leq \ker\lambda$.
\end{shtheorem}

To state a consequence of Theorem~\ref{thm:simple}, for a finite group $G$, define
\[
P(G)= \min\{ d \mid G \preccurlyeq S_d \}.
\]

\begin{shcorollary} \label{cor:simple}
Let $G$ and $H$ be as in Theorem $\ref{thm:simple}$, and let $\lambda\: H \to \PGL_m(k)$ be an irreducible projective representation. 
Assume that $m<P(G)$, and if $\mathrm{char}\, k \neq 2$, that $m < 2^{n_G}$. Then $\ker\gamma \leq \ker\lambda$.
\end{shcorollary}

Corollary~\ref{cor:simple} follows immediately from Theorem \ref{thm:simple}, since any irreducible representation of $H$ of dimension $m<P(G)$ must be primitive. The values of $P(G)$ for $G = G(q) \in {\rm Lie}(p)$ are given by \cite[Thm. 5.2.2]{ref:KleidmanLiebeck} if $G$ is a classical group, and by \cite{ref:LiebeckSaxl} if $G$ is an exceptional group. In all cases, $P(G)$ is a polynomial in $q$ of degree at least the Lie rank of $G$ (and in most cases much larger degree).

Corollary~\ref{cor:simple} is not true without the hypothesis $m<P(G)$, as the next example illustrates.

\begin{example*} \label{ex:alternating}
Let $G = A_n$ with $n > 9$ and let $p$ be an odd prime dividing $n$. By \cite[Theorem~1.2(i)]{ref:GuralnickLiebeck19}, there exists a minimal extension $H$ of $G$ by an elementary abelian $p$-group $M$ that embeds in $S_{\frac{1}{2}pn(n-1)}$ and hence in  $\PGL_m(k)$ for $m < \frac{1}{2}pn(n-1)$. (The extension is minimal since, as noted in the proof of \cite[Theorem~1.2(i)]{ref:GuralnickLiebeck19}, $M$ is contained in the Frattini subgroup of $G$.) However, $n_G \geq n-2$, so if $p=3$, then $m < \tfrac{3}{2}n(n-1) < 2^{n_G}$.
\end{example*}

The next result determines the low-dimensional projective representations of extensions of simple groups of Lie type in defining characteristic. 

\begin{shcorollary} \label{cor:simple1}
Assume that $\mathrm{char}\, k = p > 0$. Let $G \in {\rm Lie}(p)$ and assume that $(G,p) \neq (\PSp_4(3),3)$. Let $\gamma:H \to G$ be a finite minimal extension, and $\lambda\: H \to \PGL_m(k)$ an irreducible projective representation. If $m<P(G)$, then 
$\ker\gamma \leq \ker\lambda$.
\end{shcorollary}

Corollary~\ref{cor:simple1} is proved in Section \ref{s:proofs}. The excluded case $(G,p) = (\PSp_4(3),3)$ is a genuine exception, as there is a minimal extension of the form $2^{6}.G \cong 2^6.\Omega^-_6(2)$ that has a faithful irreducible projective representation of dimension 8 in characteristic 3, whereas $P(G) = 27$.

As mentioned above, if $G = G(q) \in {\rm Lie}(p)$, then $P(G)$ is a polynomial in $q$ of degree at least the Lie rank of $G$, whereas $G$ has irreducible representations of much smaller dimension than this, so Corollary \ref{cor:simple1} is quite effective in many cases. We illustrate with some examples.

\begin{example*} \label{ex:Liep} \quad 
\begin{enumerate}
\item Let $G = \PSL_d(q)$, excluding $\PSL_4(2) \cong A_8$, and let $\gamma\:H \to G$ be a finite minimal extension. Suppose
$\lambda\: H \to \PGL_m(k)$ an irreducible projective representation of dimension $m \leq \frac{1}{2}d(d+1)$, where ${\rm char}\,k=p$ and $q=p^f$. If $d = 2$, then $P(G) > 3 = \frac{1}{2}d(d+1)$, and if $d \geq 3$, then $P(G) = \frac{q^d-1}{q-1} > \frac{1}{2}d(d+1)$. Hence $m<P(G)$, so Corollary \ref{cor:simple1} implies that $\ker\gamma \leq \ker\lambda$, which means that $\lambda$ is the lift of a projective representation of $G$. The irreducible projective representations of $G$ of dimension at most $\frac{1}{2}d(d+1)$ are given in \cite[Proposition~5.4.11]{ref:KleidmanLiebeck}: $\lambda$ (or its dual) is the lift of the natural representation of $G$ of dimension $d$, its alternating or symmetric square of dimension $\frac{1}{2}d(d\pm 1)$, or its alternating cube of dimension 20 if $d=6$.
\item If $G = E_8(q)$, then $P(G)$ is the index of the largest parabolic subgroup, which is a polynomial in $q$ of degree 57. Hence if $H$ is a finite minimal extension of $G$, then every irreducible projective representation over $k$ of dimension at most $100000$ is the lift of one of the representations of $G$ given in \cite[Table A.53]{ref:Lubeck}.
\end{enumerate}
\end{example*}

In our next result, we extend considerations to extensions of characteristically simple groups, not just simple groups. To state it, we need to define a variant of $n_G$:
\[
n'_G= \min\{ n \mid G \preccurlyeq \GL_{2n}(2) \text{ irreducible} \}.
\]
It is clear that $n'_G \leq n_G$, but, as demonstrated in Lemma~\ref{lem:multiplicative}, $n'_G$ is easier to work with.

\begin{shtheorem} \label{thm:primitive}
Let $G = T^\ell$ for a finite nonabelian simple group $T$ and a positive integer $\ell$. Let $\gamma\: H \to G$ be a finite minimal extension. Let $\lambda\: H \to \PGL_m(k)$ be a faithful primitive projective representation. Assume that $m < \ell \cdot 2^{n'_T}$ if $\mathrm{char}\, k \neq 2$. Then $\gamma$ is an isomorphism.
\end{shtheorem}

The following corollary is immediate from the theorem.

\begin{shcorollary} \label{cor:primitive}
Let $G = T^\ell$ for a finite nonabelian simple group $T$ and a positive integer $\ell$. Let $\gamma\: H \to G$ be a finite proper minimal extension.
\begin{enumerate}
\item If $\mathrm{char}\, k = 2$, then $H$ has no faithful irreducible primitive projective representations.
\item  If $\mathrm{char}\, k \neq 2$, then every faithful primitive projective representation of $H$ has dimension at least $\ell \cdot 2^{n'_T}$.
\end{enumerate} 
\end{shcorollary}

We expect these results to have applications. Indeed, Theorem~\ref{thm:primitive} has already been applied in a recent paper of Ellis and Harper \cite{ref:EllisHarper}.

\begin{acknowledgements*}
The first author is an EPSRC Postdoctoral Fellow (EP/X011879/1). In order to meet institutional and research funder open access requirements, any accepted manuscript arising shall be open access under a Creative Commons Attribution (CC BY) reuse licence with zero embargo.
\end{acknowledgements*}

\section{Preliminaries} \label{s:prelims}

\subsection{\boldmath Symplectic-type $r$-groups} \label{ss:p_r-groups}

We use this first preliminary section to collect together some key results from \cite[Section~2]{ref:FeitTits78}.

Let $V = \F_r^d$ where $r$ is prime and let $f$ be an alternating form on $V$. If $r=2$, then let $Q$ be a quadratic form on $V$ with bilinear form $f$. Write $\dim (V/\mathrm{rad}\, V) = 2n$. As noted in \cite[(2.1)]{ref:FeitTits78}, up to isomorphism, there exists a unique central extension of groups
\[
0 \to \F_r \to R \xrightarrow{\pi} V \to 0
\] 
such that for all $x, y \in R$, $[x,y] = f(x\pi,y\pi)$ and $x^r = 1$ if $r \neq 2$ and $x^r = Q(x\pi)$ if $r=2$. Now assume that $Z(R)$ is cyclic, which means that 
\begin{enumerate}[1.]
\item if $r \neq 2$, then $f$ is nondegenerate and the isometry group of $f$ is $X = \Sp_{2n}(r)$
\item if $r = 2$, then $Q$ is nondegenerate and the isometry group $X$ of $Q$ is as follows
\begin{enumerate}[(a)]
\item $f$ is nondegenerate and $X = \Or^\e_{2n}(2)$ where $\e \in \{+,-\}$ is the sign of $Q$
\item $f$ has defect $1$ and $X = \Or_{2n+1}(2) \cong \Sp_{2n}(2)$.
\end{enumerate}
\end{enumerate}
According to these cases, we will denote $R$ by 
\begin{equation} \label{eq:e}
\text{(1)}\quad r^{1+2n} \qquad \text{(2a)}\quad 2^{1+2n}_\pm \qquad \text{(2b)}\quad 4 \circ 2^{1+2n}.
\end{equation}
These groups have a straightforward characterisation \cite[Lemma~2.6]{ref:FeitTits78}.

\begin{lemma} \label{lem:e_char}
Let $r$ be prime and let $R$ be a nonabelian $r$-group all of whose proper characteristic subgroups are cyclic and central. Then $R$ is isomorphic to a group in \eqref{eq:e} for some $n$.
\end{lemma}

Let $R$ be a group in \eqref{eq:e}. Write $I = \Inn(R)$ and $A = C_{\Aut(R)}(Z(R))$. Continue to write $X$ for the associated isometry group. The following is recorded in \cite[(2.2) \& (2.3)]{ref:FeitTits78}.

\begin{lemma} \label{lem:e_extension}
There is a short exact sequence $1 \to I \to A \to X \to 1$, which splits if $r \neq 2$.
\end{lemma}

The following is \cite[(2.4) \& Proposition~2.5]{ref:FeitTits78}.

\begin{lemma} \label{lem:e_reps}
Let $k$ be an algebraically closed field with $\mathrm{char}\, k \neq r$. Let $\tau\:R \to \GL_m(k)$ be a faithful irreducible representation. Then $m=r^n$ and $N_{\GL_m(k)}(R\tau)$ is an extension of $A$ by $Z(\GL_m(k))$.
\end{lemma}

The following is extracted from \cite[(3.4)]{ref:FeitTits78}.

\begin{proposition} \label{prop:feit-tits}
Let $k$ be an algebraically closed field, let $m \geq 2$ be an integer, let $H \leq \GL_m(k)$ and let $N$ be a normal subgroup of $H$ such that
\begin{enumerate}
\item $N$ is nilpotent
\item $N$ contains the subgroup of $Z(\GL_m(k))$ of order $4$ if $\mathrm{char}\, k \neq 2$
\item $N$ is an irreducible subgroup of $\GL_m(k)$
\item for all $M \leq N$ such that $M \leqn H$, either $M$ is an irreducible subgroup of $\GL_m(k)$ or $M$ is a cyclic subgroup of $Z(H)$.
\end{enumerate}
Then there exists a prime $r \neq \mathrm{char}\, k$ and a positive integer $n$ such that $m=r^n$ and there exists a faithful irreducible representation $H/N \to \Sp_{2n}(r)$. 
\end{proposition}

\begin{proof}
Let $E_0$ be a minimal noncentral normal subgroup of $H$ contained in $N$. Since $N$ is nilpotent, $E_0$ is a $r$-group for some prime $r$. If $\mathrm{char}\, k = p > 0$, then $N$ has no normal $p$-subgroup since $H$ is irreducible, but $N$ is nilpotent, which means that $N$ is a $p'$-group, so $r \neq p$. If $r \neq 2$, then let $E = E_0$, and if $r = 2$, then $E = E_0Z_4$, where $Z_4$ is the subgroup of $Z(\GL_m(k))$ of order $4$. Since $E$ is noncentral, the hypotheses in the statement ensure that $E$ is irreducible and hence $E$ is nonabelian. Note that every proper characteristic subgroup of $E$ is cyclic and central. Therefore, by Lemma~\ref{lem:e_char}, $E$ is isomorphic to a group in \eqref{eq:e} for some $n$. In fact, the choice of $E$ ensures that $E = r^{1+2n}$ if $r \neq 2$ and $E = 4 \circ 2^{1+2n}$ if $r = 2$. Since $E$ is irreducible, Lemma~\ref{lem:e_reps} implies that $m = r^n$. 

Let $Y = C_{N}(E)$. Since $E$ is irreducible, by Schur's Lemma, $Y \leq Z(\GL_m(k))$. In particular, $Y = Z(N)$. Let $I = \Inn(E)$ and $A = C_{\Aut(E)}(Z(E))$. Let $\widetilde{A} = N_{\GL_m(k)}(E)$, noting that, by Lemma~\ref{lem:e_reps}, we have $\widetilde{A}/Z(\GL_m(k)) = A$. In particular, we have 
\begin{gather*}
      EY \leq N  \leq H  \leq \widetilde{A} \leq \GL_m(k)
\end{gather*}
and $I = EY/Y \leq A \leq \PGL_m(k)$. Since $\widetilde{A}/EY \cong A/I \cong \Sp_{2n}(r)$, we obtain a faithful representation $\mu\: H/EY \to \Sp_{2n}(r)$. 

We claim that $(H/EY)\mu$ is irreducible. To see this, let $U$ be a proper subgroup of $EY/Y$ that is normalised by $H$. Noting that $EY/Y \cong E_0/(E_0 \cap Y)$, let $U_0$ be the corresponding subgroup of $E_0$, which is normalised by $H$. Since $E_0$ was chosen to be a minimal noncentral normal subgroup of $H$ contained in $N$, $U_0$ is central in $H$. In particular, $U$ is trivial.

We claim that $N = EY$. Since $N$ is nilpotent, $N/Y$ is an $r$-group, so $N/EY$ is an $r$-group too. However, since $(H/EY)\mu$ is irreducible $H/EY$ has no nontrivial normal $r$-subgroup, so $N = EY$. Therefore, $\mu\: H/N \to \Sp_{2n}(r)$ is a faithful irreducible representation, as required.
\end{proof}

\subsection{Subdirect products} \label{ss:p_subdirect}

We conclude this preliminary section with a technical lemma regarding subdirect products.

\begin{lemma} \label{lem:subdirect}
Let $H$ be a subdirect product of $H_1 \times H_2$ with projections $\pi_1\: H \to H_1$ and $\pi_2\: H \to H_2$. Assume that $N$ is a soluble normal subgroup of $H$ such that $H/N \cong T^\ell$ where $T$ is a nonabelian simple group and $\ell$ is a positive integer. Then there exist integers $\ell_1, \ell_2 \geq 0$ satisfying $\ell_1 + \ell_2 \geq \ell$ such that $H_1/N\pi_1 \cong T^{\ell_1}$ and $H_2/N\pi_2 \cong T^{\ell_2}$.
\end{lemma}

\begin{proof}
Let $i \in \{1,2\}$. Write $K_i = \ker\pi_i$. Since $NK_i/N \leqn H/N \cong T^\ell$, write $NK_i/N \cong T^{m_i}$ for some $m_i \leq \ell$. Now 
\[
H_i/N\pi_i \cong (H/K_i)/(NK_i/K_i) \cong H/NK_i \cong (H/N)/(NK_i/N) \cong T^{\ell-m_i},
\]
since $H/N \cong T^\ell$ and $NK_i/N \cong T^{m_i}$. Write $\ell_i = \ell-m_i$. We claim that $m_1+m_2 \leq \ell$, which implies that $\ell_1+\ell_2 \geq \ell$, as required.

To prove the claim, first note that
\[
K_i/(K_i \cap N) \cong NK_i/N \cong T^{m_i}
\] 
and $K_i \cap N$ is soluble, so $K_i \cap N$ is the soluble radical of $K_i$.  Since $K_1 \cap K_2 = 1$, we have $K_1K_2 \cong K_1 \times K_2$, which implies that $(K_1 \cap N)(K_2 \cap N)$ is the soluble radical of $K_1K_2$. Since
\[
K_1K_2/(K_1K_2 \cap N) \cong K_1K_2N/N \leq H/N \cong T^\ell
\] 
write $K_1K_2/(K_1K_2 \cap N) \cong T^m$ for some $m \leq \ell$. Since $K_1K_2 \cap N$ is soluble, $K_1K_2 \cap N$ is the soluble radical of $K_1K_2$, so $K_1K_2 \cap N = (K_1 \cap N)(K_2 \cap N)$. Therefore, 
\[
T^m \cong K_1K_2/(K_1K_2 \cap N) \cong K_1/(K_1 \cap N) \times K_2/(K_2 \cap N) \cong T^{m_1} \times T^{m_2},
\]
so $m_1 + m_2 = m \leq \ell$, as claimed.  
\end{proof}

\begin{corollary} \label{cor:subdirect}
Let $H$ be a subdirect product of $H_1 \times \cdots \times H_r$ and let $\pi_i\: H \to H_i$ be the projection onto the $i$th factor. Assume that $N$ is a soluble normal subgroup of $H$ such that $H/N \cong T^\ell$ where $T$ is a nonabelian simple group and $\ell$ is a positive integer. Then there exist nonnegative integers $\ell_1, \dots, \ell_r$ satisfying $\ell_1 + \cdots + \ell_r \geq \ell$ such that $H_i/N\pi_i \cong T^{\ell_i}$ for all $1 \leq i \leq r$.
\end{corollary}

\begin{proof}
We proceed by induction on $r$, noting that the result certainly holds when $r=1$. Now assume that $r > 1$. Let $\pi_0\:H \to H_2 \times \cdots \times H_r$ be the projection to $H_2 \times \cdots \times H_r$ and let $H_0 = \im\pi_0$. Then $H$ is a subdirect product of $H_1 \times H_0$. Therefore, by Lemma~\ref{lem:subdirect} we can fix integers $\ell_1, \ell_0 \geq 0$ satisfying $\ell_1+\ell_0 \geq \ell$ such that $H_1/N\pi_1 = T^{\ell_1}$ and $H_0/N\pi_0 = T^{\ell_0}$. Since $N\pi_0$ is soluble and $H_0$ is a subdirect product of $H_2 \times \cdots \times H_r$, by induction, there exist integers $\ell_2, \dots, \ell_r \geq 0$ satisfying $\ell_2 + \cdots + \ell_r \geq \ell_0$ such that $H_i/N\pi_i = H_i/N\pi_0\pi_i \cong T^{\ell_i}$ for all $2 \leq i \leq r$. Therefore, $\ell_1 + \cdots + \ell_r \geq \ell$ and $H_i/N\pi_i \cong T^{\ell_i}$ for all $1 \leq i \leq r$, which completes the induction.
\end{proof}

\section{Proofs} \label{s:proofs}

We now prove our main results. The proofs of Theorems~\ref{thm:simple} and~\ref{thm:primitive} are both inspired by the proof of \cite[Theorem~(II')]{ref:FeitTits78}.

\begin{proof}[{\bf Proof of Theorem~\ref{thm:simple}}]
For a contradiction, suppose otherwise, and choose $m$ minimally for a counterexample. Let $N = \ker\gamma$. 

We claim that $N$ is the unique maximal normal subgroup of $H$. To see this, let $K$ be a proper normal subgroup of $H$. Since $K < H$ we have $K\gamma < G$, so $K\gamma = 1$ since $G$ is simple. Therefore, $K \leq N$, as claimed. 

In particular, $\ker\lambda \leq N$, so $G$ is a quotient of $H\lambda$. Therefore, by replacing $H$ by $H\lambda$ we may assume that $\lambda$ is faithful. Let $\eta\: \widetilde{H} \to H$ be a finite central extension such that $\lambda$ lifts to a faithful representation $\widetilde{\lambda}\: \widetilde{H} \to \GL_m(k)$, and, if $\text{char}\, k \neq 2$, then choose $\widetilde{H}$ such that $\widetilde{H}\widetilde{\lambda}$ contains the subgroup $Z_4$ of $Z(\GL_m(k))$ of order $4$. Let $\widetilde{\gamma} = \eta\gamma$ and let $\widetilde{N} = \ker\widetilde{\gamma}$. 

We claim that $\widetilde{N}$ is nilpotent. Let $S$ be a Sylow subgroup of $N$. By the Frattini argument, $H = N_H(S)N$, so $N_H(S)\gamma = G$, which implies that $H = N_H(S)$, so $S \leqn N$. Therefore, $N$ is nilpotent, and $\widetilde{N}$, being a central extension of $N$, is also nilpotent.

Let $\widetilde{M} \leq \widetilde{N}$ such that $\widetilde{M} \leqn \widetilde{H}$. We claim that either $\widetilde{\lambda}_{\widetilde{M}}$ is irreducible or $\widetilde{M}$ is a cyclic subgroup of $Z(\GL_m(k))$. Since $\widetilde{\lambda}$ is irreducible, by Clifford's Theorem, we can write $\widetilde{\lambda}_{\widetilde{M}} = \alpha_1 \oplus \cdots \oplus \alpha_k$ where $\alpha_1, \dots, \alpha_k$ are the homogenous components of $\widetilde{\lambda}_{\widetilde{M}}$ whose respective irreducible components are pairwise nonisomorphic. Moreover, $\widetilde{H}$ acts transitively on these $k$ components, and 
so, since by hypothesis the representation $\widetilde{\lambda}$ is primitive, we have $k=1$. 
Thus $\widetilde{\lambda}|_{\widetilde{M}} = \alpha_1$, which is the direct sum of $m_1$ isomorphic irreducible representations of dimension $m_2$. Therefore, $\lambda = \lambda_1 \otimes \lambda_2$ where $\lambda_1\: H \to \PGL_{m_1}(k)$ and $\lambda_2\: H \to \PGL_{m_2}(k)$ are irreducible projective representations of $H$ (see \cite[Theorem~3]{ref:Clifford37}). Suppose that $1 < m_1 < m$ and $1 < m_2 < m$. The minimality of $m$ means that $\ker\gamma \leq \ker\lambda_1$ and $\ker\gamma \leq \ker\lambda_2$, so $\ker\gamma \leq \ker\lambda$, which is a contradiction. 

Therefore, we can assume that either $m_1=1$ or $m_2=1$, which is to say, either $\widetilde{\lambda}_{\widetilde{M}}$ is irreducible, or $\widetilde{\lambda}_{\widetilde{M}}$ a direct sum of isomorphic linear representations. In the latter case, $\widetilde{M}$ is represented by scalars, so $\widetilde{M}$ is a cyclic subgroup of $Z(\GL_m(k))$.

We claim that $\widetilde{\lambda}_{\widetilde{N}}$ is irreducible. Suppose otherwise. Then the previous paragraph implies that $\widetilde{N} \leq Z(\GL_m(k))$, so $\ker\gamma = N = \widetilde{N}\eta = 1$, which is a contradiction. 

Therefore, Proposition~\ref{prop:feit-tits} yields a faithful irreducible representation $\mu\: G \to \Sp_{2n}(r)$ where $m = r^n$ for a prime $r \neq \mathrm{char}\, k$. By Lemma~\ref{lem:e_extension}, if $r \neq 2$, then $H$ is a split extension of $G$, so $\ker\gamma = 1$, which is a contradiction. Therefore, $r = 2$, so $\mathrm{char}\, k \neq 2$ and $m = 2^n \geq 2^{n_G}$, which contradicts our assumption that $m < 2^{n_G}$.
\end{proof}

\begin{proof}[{\bf Proof of Corollary~\ref{cor:simple1}}]
Let $G \in {\rm Lie}(p)$, and let $\gamma:H \to G$ be a finite minimal extension. 
The conclusion follows from Corollary \ref{cor:simple} if $p=2$, so assume that $p\ne 2$.

We first claim that either $P(G) \leq 2^{n_G}$ or $G$ is as in the table below:
\[
\begin{array}{|l|l|l|}
\hline
G & P(G) & n_G \\
\hline
\PSL_2(17) & 18 & 4 \\
\PSp_4(3) & 27 & 3 \\
\PSU_3(3) & 28 & 3 \\
G_2(3) & 351 & 7 \\
\hline
\end{array}
\]
It is a routine matter to verify the claim. The values of $P(G)$ are given by \cite[Thm. 5.2.2]{ref:KleidmanLiebeck} if $G$ is a classical group, and by \cite{ref:LiebeckSaxl} when $G$ is an exceptional group of Lie type; while $n_G \geq \frac{1}{2}R_{p'}(G)$, lower bounds for which can be found in \cite[Table 5.3.A]{ref:KleidmanLiebeck}. Comparison of these bounds gives $P(G) \leq 2^{n_G}$ apart from some small simple groups $G$ (including those in the table above), for which the precise values of $n_G$ can be read off using \cite{ref:ModularAtlas}. The claim follows.

Given the claim, the conclusion of Corollary \ref{cor:simple1} follows from Corollary \ref{cor:simple}, provided we rule out the groups $G$ in the above table. The group $\PSp_4(3)$ is excluded by hypothesis. Now consider $G = \PSL_2(17)$ (resp. $\PSU_3(3)$, $G_2(3)$). 
From \cite{ref:ModularAtlas} we see that the nontrivial irreducible $\F_2G$-modules of dimension at most $2\log_2(P(G))$ have dimensions 8 (resp. 6, 14). Suppose $\lambda:H \to \PGL_m(k)$ is irreducible of dimension $m<P(G)$ with $\ker\gamma \not \leq \ker\lambda$.
Then it follows from the proof of Theorem \ref{thm:simple} that $K:=\ker \gamma = 2^8$ (resp. $2^6$, $2^{14}$) and $m=16$ (resp. 8, 128).
Now $H$ must be a nonsplit extension of $G$ by $K$. However, a computation in \textsc{Magma} \cite{ref:Magma} shows that $H^2(G,K)=0$ for $G = \PSL_2(17)$, $G_2(3)$, which rules out these groups. The last observation does not apply to $G = \PSU_3(3)$ (as $H^2(\PSU_3(3),2^6)) \ne 0$); however, for this case we have $H\lambda = 2^6.\PSU_3(3) < \PGL_8(9) < \PGL_8(k)$, whereas a \textsc{Magma} computation reveals that this extension $2^6.\PSU_3(3)$ splits. This completes the proof.
\end{proof}

Before proving Theorem~\ref{thm:primitive} we need a result that highlights a key property of the invariant $n_G'$ defined in the introduction.

\begin{lemma} \label{lem:multiplicative}
Let $G = T^\ell$ for a nonabelian finite simple group $T$ and $\ell \geq 1$. Then $n_G' \geq n_T' \cdot 2^{\ell-1}$.  
\end{lemma}

\begin{proof}
The definition of $n_G'$ guarantees the existence of a faithful irreducible representation $\rho\: G \to \GL_{2n_G'}(2)$. Let $V$ be the $\F_2G$-module afforded by $\rho$, let $E = \mathrm{End}_{\F_2G}(V)$ and let $e = |E:\F_2|$. Then $e$ divides $2n_G'$ and there exists a faithful absolutely irreducible representation $\rho_1 \: G \to \GL_{2n'_G/e}(2^e)$ (see \cite[Lemma~2.10.2]{ref:KleidmanLiebeck}). Since $G = T^\ell$, there exists a faithful absolutely irreducible representation $\rho_2 \: T \to \GL_n(2^e)$ where $n^\ell = 2n_G'/e$ (see \cite[Lemma~5.5.5]{ref:KleidmanLiebeck}). Let $f$ be minimal such that $\rho_2$ is expressible over the subfield $\F_{2^f} \subseteq \F_{2^e}$, and consider the corresponding faithful absolutely irreducible representation $\rho_3 \: T \to \GL_n(2^f)$. Via the field extension embedding $\GL_n(2^f) \preccurlyeq \GL_{fn}(2)$, we obtain a faithful irreducible representation $\rho_4\: T \to \GL_{fn}(2)$. Therefore, $fn \geq 2n_T'$. We now conclude that
\[
2n_T' \leq fn \leq en \leq \frac{en^\ell}{2^{\ell-1}} \leq \frac{n_G'}{2^{\ell-2}},
\]
and hence $n_G' \geq n_T' \cdot 2^{\ell-1}$, as required.
\end{proof}

\begin{proof}[{\bf Proof of Theorem~\ref{thm:primitive}}]
For a contradiction, suppose otherwise, and choose $m$ minimally for a counterexample. Let $N = \ker\gamma$.

Since $\lambda$ is faithful, by replacing $H$ with $H\lambda$, we assume that $H \leq \PGL_m(k)$. Let $\eta\: \widetilde{H} \to H$ be a finite central extension such that $\widetilde{H} \leq \GL_m(k)$, and, if $\mathrm{char}\, k \neq 2$, then choose $\widetilde{H}$ to contain the subgroup $Z_4$ of $Z(\GL_m(k))$ of order $4$. Let $\widetilde{\gamma} = \eta\gamma$ and let $\widetilde{N} = \ker\widetilde{\gamma}$. 

As in the proof of Theorem~\ref{thm:simple}, $\widetilde{N}$ is nilpotent. Let $\widetilde{M} \leq \widetilde{N}$ such that $\widetilde{M} \leqn \widetilde{H}$. We claim that either $\widetilde{M}$ is irreducible or $\widetilde{M}$ is a cyclic subgroup of $Z(\widetilde{H})$. Since $\widetilde{H}$ is primitive, by Clifford's theorem, the representation of $\widetilde{M}$ is the direct sum of $m_1$ isomorphic irreducible representations of dimension $m_2$. Therefore, $\lambda  = \lambda_1 \otimes \lambda_2$ where $\lambda_1\: H \to \PGL_{m_1}(k)$ and $\lambda_2\: H \to \PGL_{m_2}(k)$ are irreducible projective representations of $H$ (see \cite[Theorem~3]{ref:Clifford37}). Moreover, $\lambda_1$ and $\lambda_2$ are both primitive since $\lambda$ is. 

Suppose that $1 < m_1 < m$ and $1 < m_2 < m$. For $i \in \{1,2\}$, let $K_i = \ker\lambda_i$ and $H_i = H/K_i$, and note that we have a faithful primitive projective representation of $H_i$ afforded by the embedding $H_i \cong \im\lambda_i \leq \PGL_{m_i}(k)$. Let $N_i = NK_i/K_i$ and let $\gamma_i\: H_i \to H_i/N_i$ be the corresponding quotient map. Now $H$ is isomorphic to a subdirect product of $H_1 \times H_2$, so Lemma~\ref{lem:subdirect} implies that $H_i/N_i = T^{\ell_i}$ for integers $\ell_1,\ell_2 \geq 0$ such that $\ell_1 + \ell_2 \geq \ell$. Suppose that there exists $i \in \{1,2\}$ and $X < H_i$ such that $X\gamma_i = H_i/N_i$. Then $XN_i/N_i = H_i/N_i$, so $XN_i = H_i$. Since $H_i = H/K_i$, we can write $X = Y/K_i$ for some $K_i \leq Y < H$. Then $(Y/K_i) (NK_i/K_i) = H/K_i$, so $YN = H$ since $K_i \leq Y$. By assumption, this means that $Y = H$, so $X = H_i$, which is a contradiction. Therefore, for $i \in \{1,2\}$, if $X < H_i$, then $X\gamma_i < H_i/N_i$. For now assume that $\mathrm{char}\, k \neq 2$. We claim that $m_i < \ell_i \cdot 2^{n_T'}$ for $i \in \{1,2\}$. Write $\{i,j\} = \{1,2\}$. Suppose that $m_i \geq \ell_i \cdot 2^{n_T'}$. Since $m < \ell \cdot 2^{n_T'}$, we have 
\[
m_j = m/m_i < \ell/\ell_i \leq 2(\ell-\ell_i) \leq 2\ell_j. 
\]
Note that $m_j < \ell_j \cdot 2^{n_T'}$ since $n_T' \geq 1$. Therefore, the minimality of $m$ implies that $\gamma_j$ is an isomorphism, so $T^{\ell_j} = H_j$ is an irreducible subgroup of $\PGL_{m_j}(k)$. Writing $d_p(T)$ for the smallest dimension $d$ of an irreducible projective representation $T \to \PGL_d(k)$, by \cite[Proposition~5.5.7(ii)]{ref:KleidmanLiebeck}, we have $m_j \geq d_p(T)^{\ell_j} \geq 2^{\ell_j} \geq 2\ell_j$, which contradicts the displayed inequality. Therefore, $m_i < \ell_i \cdot 2^{n_T'}$ as claimed. Since we have verified all of the conditions of the statement, the minimality of $m$ implies that $\gamma_1$ and $\gamma_2$ are isomorphisms. This implies that $N \leq K_1$ and $N \leq K_2$, so $N \leq K_1 \cap K_2 = 1$, which is a contradiction.

Therefore, we can assume that either $m_1=1$ or $m_2=1$, which is to say, either $\widetilde{M}$ is irreducible, or $\widetilde{M}$ a direct sum of isomorphic linear representations. In the latter case, $\widetilde{M}$ is represented by scalars, so $\widetilde{M}$ is a cyclic subgroup of $Z(\GL_m(k))$.

We claim that $\widetilde{\lambda}_{\widetilde{N}}$ is irreducible. Suppose otherwise. Then the previous paragraph implies that $\widetilde{N} \leq Z(\GL_m(k))$, so $\ker\gamma = N = \widetilde{N}\eta = 1$, which is a contradiction. 

Therefore, Proposition~\ref{prop:feit-tits} yields a faithful irreducible representation $\mu\: G \to \Sp_{2n}(r)$ where $m = r^n$ for a prime $r \neq \mathrm{char}\, k$. By Lemma~\ref{lem:e_extension}, if $r \neq 2$, then $H$ is a split extension of $G$, so $\ker\gamma = 1$, which is a contradiction. Therefore, $r = 2$, so $\mathrm{char}\, k \neq 2$ and, using Lemma~\ref{lem:multiplicative}, 
\[
m = 2^n \geq 2^{n_G} \geq 2^{n_G'} = 2^{2^{\ell-1}(n_T')^\ell} \geq \ell \cdot 2^{n_T'},
\] 
which contradicts our assumption that $m < \ell \cdot 2^{n_T'}$.
\end{proof}

\vspace{5pt}

\begin{multicols}{2}
\noindent Scott Harper \newline
School of Mathematics and Statistics \newline
University of St Andrews \newline
St Andrews, KY16 9SS, UK \newline
\texttt{scott.harper@st-andrews.ac.uk}

\noindent Martin W. Liebeck \newline
Department of Mathematics \newline
Imperial College London \newline
London, SW7 2BZ, UK \newline
\texttt{m.liebeck@imperial.ac.uk}
\end{multicols}

\end{document}